\documentclass[a4paper,12pt]{article}
\usepackage{amsmath,amssymb,amsthm}
\usepackage[top=2.5cm,bottom=2cm,left=2.5cm,right=2.5cm]{geometry}
\usepackage[margin=1cm,%
           font=small,%
           format=hang,%
           labelsep=period,%
           labelfont=bf]{caption}

\usepackage{graphicx}
\usepackage{epstopdf}

\newtheorem{proposition}{Proposition}

\newtheorem{corollary}{Corollary}
\newtheorem{problem}{Problem}

\begin{document}

\thispagestyle{empty}

\begin{center}

{\Large \bf Some results on the Wiener index related \\[1mm] to the \v{S}olt\'{e}s problem of graphs}

\vspace{3mm}

Andrey A. Dobrynin$^{a}$ and Konstantin V. Vorob'ev$^{a,b}$

\vspace{3mm}

{\it $^a$Sobolev Institute of Mathematics, Siberian Branch of the 
Russian Academy of \\ Sciences,  Novosibirsk, 630090, Russia, dobr@math.nsc.ru \\
$^b$Institute of Mathematics and Informatics, Bulgarian Academy of Sciences, \\
 Sofia 1113, Bulgaria, konstantin.vorobev@gmail.com }

\end{center}

\vspace{3mm}

\begin{abstract}
The Wiener index, $W(G)$,  of a connected graph $G$ is the sum of distances
between its vertices.
In 2021, Akhmejanova et al. posed the problem of finding  graphs $G$
with large
$R_m(G)= |\{v\in V(G)\,|\,W(G)-W(G-v)=m \in \mathbb{Z} \}|/ |V(G)|$.
It is shown that there is a graph $G$ with  $R_m(G) > 1/2$
 for any integer $m \ge 0$.
In particular, there is a regular graph of even degree
with this property  for any odd $m \ge 1$.
The proposed approach allows to construct new families of graphs $G$ with
$R_0(G) \rightarrow 1/2$ when the order of $G$ increases.
\end{abstract}



\section{Introduction}
The graphs $G$ we consider are undirected, simple and connected.
The vertex set of $G$ is denoted by $V(G)$ and the cardinality of $V(G)$
is called the order of $G$.
A graph obtained by removing a vertex $v$ from $G$
is denoted by $G-v$.
The distance $d(u,v)$ between vertices
$u,v \in V(G)$ is the number of edges on a shortest $(u,v)$-path
in $G$.
The transmission of a vertex $v \in V(G)$ is the sum
of distances from $v$ to all vertices of $G$,
$tr(v)=\sum_{u\in V(G)} d(v,u)$.
A half of the sum of all vertex transmissions defines the Wiener index
of a graph,
$$
W(G)=\frac{1}{2} \sum_{v\in V(G)} tr(v).
$$
This distance graph invariant and its numerous modifications and
generalizations are intensively studied in theoretical
and mathematical chemistry
\cite{Bonch91,Dobr01,Dobr02,Gutm86,Knor16,Niko95,Polan86,Tode00,Trin83}.
In particular, the Wiener index has found numerous applications
in the modelling of physico-chemical and biological
properties of organic molecules
\cite{Bonch91,Gutm86,Rouv83,Trin83,Wien47}.

If graph $G-v$ is connected, then we are interested in estimating
the number of vertices $v$ having coinciding differences
$\Delta_v(G)=W(G)-W(G-v)$.
Let
$R_m(G)= |\{v\in V(G)\,|\,\Delta_v(G)=m \in \mathbb{Z} \}|/|V(G)|$.
In 2023, Akhmejanova et al.
\cite{Akhm2023}
posed the problem of finding  graphs $G$ with large $R_m(G)$.
They constructed graphs $G$ with  $R_m(G) < 1/2$ and
$R_m(G) \rightarrow 1/2$ when the order of $G$ tends to infinity
for any integer $m$.
Special attention of researchers is attracted by the case
$\Delta_v(G)=0$.
In 1991, \v{S}olt\'{e}s
\cite{Solt91}
posed the following problem:
find all graphs $G$ for which equality $\Delta_v(G)=0$ holds for all
vertices $v$ of $G$.
A vertex $v$  is called a \v{S}olt\'{e}s vertex  if $\Delta_v(G)=0$.
Graphs with $R_0(G)=1$ gives a solution of the \v{S}olt\'{e}s problem.
Simple cycle $C_{11}$ is the unique known example of such graphs.
The second known maximum value of $R_0(G)$ is 2/3
and it is achieved on graph $G$ shown in Fig.~\ref{Fig1}
\cite{Akhm2023}.
Failures in building new \v{S}olt\'{e}s graphs led to the
formulation of various relaxed problems
\cite{Akhm2023,Bas23,Bok19,Bok21,Dobr03,Hu21,Knor18-1,Knor18-2,
Knor23,Majs18,Spir22}.

\begin{figure}[h]
\center{\includegraphics[width=0.25\linewidth]{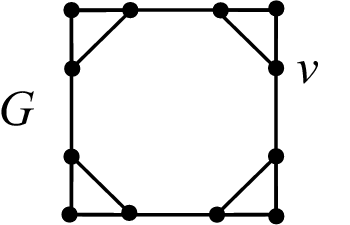}}
\caption{Graph $G$ with $W(G)=W(G-v)$ and $R_0(G)=2/3 \approx 0.667$.}
\label{Fig1}
\end{figure}

In this paper, we construct graphs $G$ with $R_m(G) > 1/2$
for any integer $m~\ge~0$. In particular, some of them are
regular graphs for odd $m \ge 1$.
The proposed approach also allows to construct new families of
graphs $G$ with $R_0(G) < 1/2$ and $R_0(G) \rightarrow 1/2$
when the order of $G$ increases.

\section{General construction}

The absence of the transmission of a vertex $v$ removed from a graph $G$
must be compensated by an increase of distances between vertices in
graph $G-v$.
Known examples show that a promising way to find graphs
having large number of \v{S}olt\'{e}s vertices
is to remove a vertex $v$ from long cycles without chords.

Consider graph $H=H(n,k,l,n_0,t_0)$ shown in
Fig.~\ref{Fig2},
where $n \ge 3$, $k \ge 2$, and $l \ge 1$.
It consists of $k$ copies of the simple cycle $C_n$
whose vertices are identified with vertices of one part
of $n$ copies of the complete bipartite graph $K_{k,l}$.
The vertices of the other part of every $K_{k,l}$ are included in an arbitrary
graph $F$ of order $n_0$ in the same way.
The order of $H$ is $n(k + n_0)$.
Let $C$ be a union of $k$ cycles $C_n$ of $H$, $|V(C)|=kn$.
Denote by $F_v$ a copy of graph $F$ whose vertices are adjacent with a vertex $v$ of $C$.
Since every $v \in C$ is adjacent to all vertices of
$V(F_v) \cap V(K_{k,l})$, the vertex set of $C$ form an orbit
of the automorphism group of $H$.
Let $t_0$ be the transmission of a vertex $v$ in the induced subgraph
with the vertex set $V(F_v) \cup \{v\}$.

\begin{figure}[t]
\center{\includegraphics[width=0.6\linewidth]{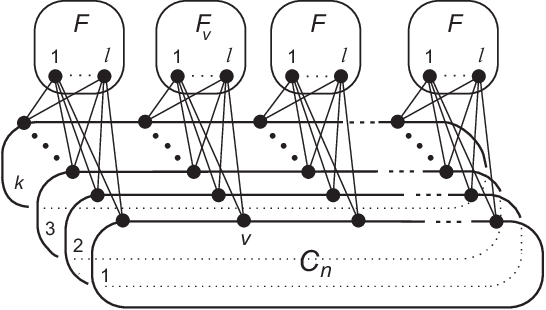}}
\caption{Graph $H=H(n,k,l,n_0,t_0)$.}
\label{Fig2}
\end{figure}

\begin{proposition} \label{Prop1}
Let $v$ be an arbitrary vertex of $C$ and $m=\Delta_v(H)$. Then
$$
R_m(H) \geq \frac{k}{k + n_0},
$$
where
$$
\Delta_v(H)= n(2k+t_0+n_0+2) - \frac{n^2}{4}(n_0-k+2) -
  \begin{cases}
     2n_0+8,               & \text{if $n$ is even,} \\[1mm]
     \frac{k+11n_0+34}{4}, & \text{if $n$ is odd.}
  \end{cases}
$$
\end{proposition}

\begin{proof}
For convenience, we will calculate the value of $\Delta'_v(H) = -\Delta_v(H)$.
The difference $\Delta'_v(H)= W(H-v)-W(H)$ can be presented as
\begin{equation}{\label{eq1}}
 \Delta'_v(H)= \frac{1}{2} \sum_{u,w\in V(H-v)} [d_{H-v}(u,w)-d_{H}(u,w)]  -tr_H(v).
\end{equation}

By direct calculations, we have
\begin{equation}{\label{eq2}}
tr_H(v)= 	
\begin{cases}
 \frac{n^2}{4}(k+n_0)+n(2k+t_0-2),   & \text{if $n$ is even,}\\[1mm]
 \frac{n^2-1}{4}(k+n_0)+n(2k+t_0-2), & \text{if $n$ is odd.}
\end{cases}
\end{equation}
	
The next step is to find all pairs of vertices of graph $H$ (except $v$)
such that the distance between them changes after removing vertex $v$ (equivalently,
all shortest paths between pairs of vertices containing $v$).
Then the summation of the distance differences will give $\Delta'_v(H)$.

Denote by $P$ the simple path of order $n-1$ obtained by removing vertex $v=u_0$
from an arbitrary cycle $(u_0,u_1,u_2,\dots,u_{n-1},u_0)$ of $C$ (see Fig.~\ref{Fig3}).
By construction of $H-u_0$, the shortest paths between vertices $x,y \in V(H-P)$
do not contain  vertex $v$.
The same is also true if the first vertex belongs to $P$ and the second
one belongs to $C-P$.
Then it is sufficient to consider the following three cases
for calculating difference between vertex distances.

\begin{figure}[h]
\center{\includegraphics[width=\linewidth]{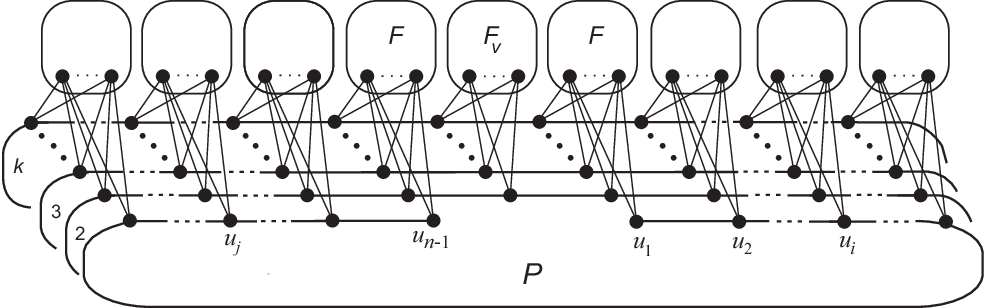}}
\caption{Graph $H-u_0$.}
\label{Fig3}
\end{figure}

\textbf{Case 1}.
Two vertices belong to the path $P$.
Because of symmetry of $P$, vertex $u_i$, $1\leq i<\frac{n}{2}-1$,
can be chosen as the first vertex (some vertices opposite to $u_0$
in the cycle preserve distances to vertices of $H-v$).
Every shortest path between $u_i$ and vertex $u_j$ of $P$ contains $u_0$ in $H$ if and only if
$\frac{n}{2}+i+1\leq j\leq n-1$ for even $n$ and $\frac{n-1}{2}+i+1\leq j\leq n-1$ for odd $n$.
By choice of $i$ and $j$, $d_H(u_i,u_j)=n+i-j$.
After removing $u_0$, there are two possible types of shortest paths between $u_i$ and $u_j$.
In the first case, we go from $u_i$ to $u_j$ by remaining vertices of $P$ and the distance is $j-i$.
In the second case, we move to another $n$-cycle, repeat there the corresponding
shortest path and then come back to $P$ which gives the distance $n+i-j-4$.
Summing all these distance differences, we have
$$
\sum_{i=1}^{\frac{n-4}{2}}\sum_{j=\frac{n}{2}+i+1}^{n-1}\min\{4,2j-2i-n\}=\frac{n^2-8n+16}{2},
\text{   if $n$ is even},
$$
$$
\sum_{i=1}^{\frac{n-3}{2}}\sum_{j=\frac{n-1}{2}+i+1}^{n-1}\min\{4,2j-2i-n\}=\frac{n^2-8n+17}{2},
\text{   if $n$ is odd}.
$$

\textbf{Case 2}.
One vertex belongs to $P$ and another one belongs to graph $F_v$.
Evidently, the distance from the first vertex to any vertex of $F_v$
is increased by two after removing $v$ (one need to move to another
$n$-cycle). Consequently, the total difference in this case equals $2n_0(n-1)$.

\textbf{Case 3}.
One vertex belongs to $P$ and another one belongs to some of $n-1$ graphs $F$ (except $F_v$).
By arguments similar to the previous cases, the sum of distance differences equals
$$
2n_0\sum_{i=1}^{\frac{n-4}{2}}\sum_{j=\frac{n}{2}+i+1}^{n-1}\min\{2,2j-2i-n\}=\frac{n_0(n^2-6n+8)}{2},
\text{   if $n$ is even},
$$
$$
2n_0\sum_{i=1}^{\frac{n-3}{2}}\sum_{j=\frac{n-1}{2}+i+1}^{n-1}\min\{2,2j-2i-n\}=\frac{n_0(n^2-6n+9)}{2},
\text{   if $n$ is odd}.
$$

Substituting expression (\ref{eq2}) and the obtained differences
into equation (\ref{eq1}), we get the required value
of $\Delta_v(H)$. 	
Since graphs $F$ also may contain vertices $w$ with  $\Delta_w(H)=m$, 	
$R_m(H) \geq k/(k + n_0)$.
\end{proof}

\section{Graphs with $\Delta_v(G)=m$}

Consider graph $G=H(n,k,l,l,l)$ of order $n(k + l)$ shown in Fig.~\ref{Fig4}.
All vertices of $C$ have degree $l+2$ while all the other
vertices of $G$ have degree $k$, i.e.,
graphs $F$ of order $l$ have no edges.
From Proposition~\ref{Prop1}, we immediately obtain the following result.

\begin{proposition} \label{Prop2}
Let $v$ be an arbitrary vertex of $C$ and $n=16m+95$, $k=m+6$, and $l=m+5$.
Then for any $m \ge 0$, $\Delta_v(G)=m$  and
$$
R_m(G) = \frac{1}{2 - \frac{1}{m+6}},
$$
that is $R_m(G) > 1/2$ for any $m\ge 0$ and
$R_m(G) \rightarrow 1/2$ when $m \rightarrow \infty$.
\end{proposition}

\begin{figure}[t]
\center{\includegraphics[width=0.6\linewidth]{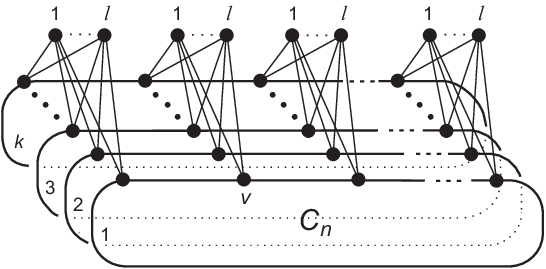}}
\caption{Graph $G=H(n,k,l,l,l)$ with  $\Delta_v(G)=m$.}
\label{Fig4}
\end{figure}

Table~\ref{Tab1} contains parameters of the initial graphs of the constructed family.
The first graph has maximal $R_0(G)=6/11 \approx 0.545$
and the other graphs satisfy condition $0.5 < R_m(G) < 0.545$.
Note that graphs $G$ described in
\cite{Akhm2023}
have $\Delta_v(G)=m$ and  $R_m(G) < 0.5$ for all $m \ge 7$.

\begin{table}[h]
\centering
\caption{Parameters of graphs $G$ having $R_m > 1/2$.} \label{Tab1}
\begin{tabular}{rrrrrrr} \hline
$m$ & $n$  & $k$ & $l$ & $R_m$  & $|V(G)|$ & $W$ \ \ \ \   \\ \hline
0   &  95  &  6  &  5  &  0.545 & 1045 \ & 13733010  \\
1   &  111 &  7  &  6  &  0.538 & 1443 \ & 30366714  \\
2   &  127 &  8  &  7  &  0.533 & 1905 \ & 60203080  \\
3   &  143 &  9  &  8  &  0.529 & 2431 \ & 109884060 \\
4   &  159 &  10 &  9  &  0.526 & 3021 \ & 187977750 \\
5   &  175 &  11 &  10 &  0.523 & 3675 \ & 305224150 \\
\hline
\end{tabular}
\end{table}

Let graph $G'$ be obtained from graph $G$ of Proposition~\ref{Prop2} by
inserting $s$ edges between vertices of every graph $F$
in the same way.
It is clear that every inserted edge decreases distances between two vertices of $F$ from 2 to 1.
Since the shortest paths between other vertices of $G'$ do not contain inserted edges,
$W(G') = W(G) - ns$ and $W(G'-v) = W(G-v) - ns$.
This implies $W(G')=W(G'-v)$.

\begin{corollary}
For every $m \ge 0$, $R_m(G')=  R_m(G)$.
\end{corollary}

It is also possible to construct regular graphs with this property for even $l$.
Let graph $G'$ be obtained from graph $G$ by inserting a perfect matching into $F$.
Then $m$ is odd and $G'$ is a regular graph.

\begin{corollary} \label{m_reg}
Let $v$ be an arbitrary vertex of $C$. Then for any odd
$m \ge 1$, there exists a $(m+7)$-regular graph $G$ having
$$
R_m(G) = \frac{1}{2- \frac{1}{m+6}},
$$
that is $R_m(G) > 1/2$ for any odd $m \ge 1$ and
$R_m(G) \rightarrow 1/2$ when $m \rightarrow \infty$.
\end{corollary}

It will be interesting to find such regular graphs
for small vertex degrees, especially, for cubic and quartic graphs
that represent chemical graphs.

\section{Graphs with $\Delta_v(G)=0$}

The proposed construction of graph $H$ makes it possible to build
new families of graphs $G$  with $R_0(G) < 1/2$ and $R_0(G) \rightarrow 1/2$
when the order of $G$ increases. Similar results were obtained in
\cite{Akhm2023}
for graphs of a different structure.

Let $G = H(2k+24,k,1,k+8,2k+51)$ be an infinite family of graphs
of order $n(2k + 8)$ shown in
Fig.~\ref{Fig5}.
Here graph $F$ consists of the star with $k-1$ branches and
the attached simple path of order 9.
Using Proposition~\ref{Prop1}, we have the following result.

\begin{figure}[h]
\center{\includegraphics[width=0.6\linewidth]{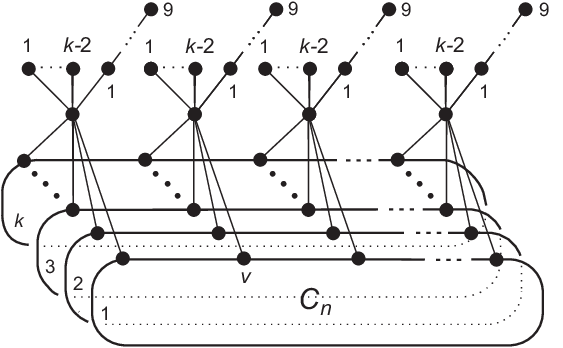}}
\caption{The family of graphs $G$ with $\Delta_v(G)=0$.}
\label{Fig5}
\end{figure}

\begin{proposition} \label{Prop3}
Let $v$ be an arbitrary vertex of $C$.
Then for any $k \ge 2$,
$$
R_0(G) = \frac{1}{2 + \frac{8}{k}},
$$
that is  $R_0(G) < 1/2$ for any $k\ge 2$ and
$R_0(G) \rightarrow 1/2$
when $k \rightarrow \infty$.
\end{proposition}

Table~\ref{Tab2} contains parameters of the initial members of the constructed family.

\begin{table}[h]
\centering
\caption{Parameters of graphs $G$ with $R_0(G) \rightarrow 1/2$.} \label{Tab2}
\begin{tabular}{rrrrrr} \hline
$k$ & $n$ & $n_0$ & $R_0$  & $|V(G)|$ & $W$ \ \ \   \\ \hline
2   &  28 &  10  &  0.167 &  336  \ \ \  &  903364 \\
3   &  30 &  11  &  0.214 &  420   \ \ \ &  1372890 \\
4   &  32 &  12  &  0.250 &  512  \ \ \ &  2013728 \\
5   &  34 &  13  &  0.278 &  612  \ \ \ &  2868682 \\
6   &  36 &  14  &  0.300 &  720  \ \ \ &  3987180 \\
7   &  38 &  15  &  0.318 &  836  \ \ \ &  5425754 \\
\hline
\end{tabular}
\end{table}

Let graph $G'$ be obtained from $G$ by inserting $s$ edges between
$k-2$ pendent vertices of the star of $F$ in the same way.
It is not hard to see that the inserted edges equally decrease
Wiener indices of graphs $G$ and $G-v$. Namely,
$W(G') = W(G) - ns$ and $W(G'-v) = W(G-v) - ns$.
Then $W(G')=W(G'-v)$.

\begin{corollary}
Proposition~\ref{Prop3} is valid for graphs $G'$.
\end{corollary}

Consider the family of graphs $G= H(n,k,1,k+13,2k+61)$
of order $n(2k+13)$ where $n=(4k+59)/3$ for $k=3i+1$, $i \ge 1$
(see Fig.~\ref{Fig6}).
Graph $F$ consists of the star with $k-1$ branches and
the attached simple cycle $C_{14}$.

\begin{figure}[h!]
\center{\includegraphics[width=0.6\linewidth]{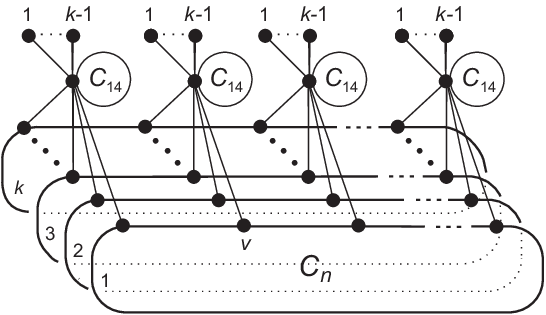}}
\caption{The family of graphs $G$ with $\Delta_v(G)=0$.}
\label{Fig6}
\end{figure}

\begin{proposition} \label{Prop4}
Let $v$ be an arbitrary vertex of $C$. Then for any $k \ge 4$
$$
R_0(G) = \frac{1}{2 + \frac{13}{k}},
$$
that is $R_0(G) < 1/2$ for any $k\ge 4$ and
$R_0(G) \rightarrow 1/2$ when $k \rightarrow \infty$.
\end{proposition}

Table~\ref{Tab3} contains parameters of the initial graphs of the considered family.

\begin{table}[h!]
\centering
\caption{Parameters of graphs $G$ with $R_0(G) \rightarrow 1/2$.} \label{Tab3}
\begin{tabular}{rrrrrr} \hline
$k$ & $n$ & $n_0$ & $R_0$  & $|V(G)|$ & $W$ \ \ \   \\ \hline
4   &  25 &  3  &  0.190 &  525 \ \  &  1757200 \\
7   &  29 &  6  &  0.259 &  783 \ \   &  3936431 \\
10  &  33 &  9  &  0.303 &  1089\ \   &  7867266 \\
13  &  37 &  12 &  0.333 &  1443\ \   & 14443801 \\
16  &  41 &  15 &  0.356 &  1845\ \   & 24824516 \\
19  &  45 &  18 &  0.373 &  2295\ \   & 40466835 \\
\hline
\end{tabular}
\end{table}

In conclusion, we present two graphs having maximum $R_0(G)$
among all graphs constructed here.
Consider graph $G=H(71,4,1,3,5)$ of order 497 shown in Fig.~\ref{Fig7}.
Then $R_0(G)=4/7 \approx 0.571$ and $W(G)=W(G-v)=2427916$.
If every two pendent vertices at distance 2 are joint by an edge,
then the resulting graph $G'$  has the same $R_0(G')=4/7$
and $W(G')=W(G'-v)=W(G)-71=2427845$.

\begin{figure}[h]
\center{\includegraphics[width=0.6\linewidth]{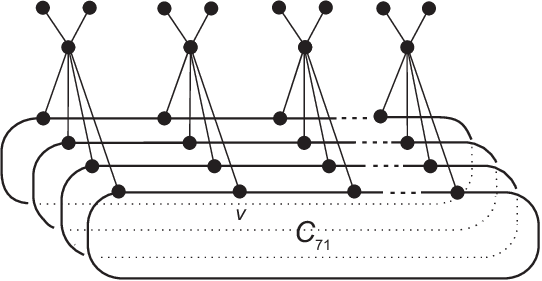}}
\caption{Graph $G$ with $R_0(G) = 4/7$.}
\label{Fig7}
\end{figure}

Our search shows that it is quite difficult to construct graphs
$G$ even with $R_0(G) > 1/2$. It would be interesting to
find graphs with a large number of \v{S}olt\'{e}s vertices.

\begin{problem}
Find graphs $G$ for which $2/3 \le R_0(G) < 1$.
\end{problem}

Graph of Fig.~\ref{Fig1} is the unique known example of such graphs.

\section*{Acknowledgements}

This work was supported by the Russian Science Foundation under grant 23-21-00459.


\end{document}